\documentclass{article}
\usepackage[centering,total={410pt,610pt}]{geometry}
\usepackage{graphicx}
\usepackage[titletoc]{appendix}
\usepackage{tikz}
\usetikzlibrary{calc}
\usepackage{pgfplots}
\pgfplotsset{compat=newest}
\usepackage{appendix}
\usepackage{tikz}
\usepackage{enumerate}

\usepackage{amsmath}
\usepackage{amssymb}
\usepackage{amsthm}
\newtheorem{theorem}{Theorem}[section]

\theoremstyle{definition}

\newtheorem{example}[theorem]{Example}
\theoremstyle{remark}
\newtheorem{remark}[theorem]{Remark}

\newtheorem{assumption}[theorem]{Assumption}

\usepackage{float}
\usepackage{hyperref}
\hypersetup{colorlinks,citecolor=black,filecolor=black,linkcolor=black,urlcolor=black}
\floatstyle{ruled}
\newfloat{algorithm}{htbp}{loa}
\floatname{algorithm}{\smallskip Algorithm\smallskip}

\numberwithin{equation}{section}

\usepackage{color}

\newcommand{\R}{\mathbb{R}}

\title{On the optimal linear convergence factor of the relaxed proximal point algorithm for monotone inclusion problems}
\author{%
Guoyong Gu%
\thanks{Department of Mathematics, Nanjing University, Nanjing, 210093, China.}
\thanks{Email: ggu@nju.edu.cn. This author was supported by the NSFC grant 11671195.}
\and
Junfeng Yang\footnotemark[1]
\thanks{Email: jfyang@nju.edu.cn. This author was supported by the NSFC grant  11771208.}
}
\date{}

\begin{document}
\maketitle

\begin{abstract}
  Finding a zero of a maximal monotone operator is fundamental in convex optimization and monotone operator theory, and \emph{proximal point algorithm} (PPA) is a primary method for solving this problem. PPA converges not only globally under fairly mild conditions but also asymptotically at a fast linear rate provided that the underlying inverse operator is Lipschitz continuous at the origin. These nice convergence properties are preserved by a relaxed variant of PPA. Recently, a linear convergence bound was established in
[M. Tao, and X. M. Yuan, J. Sci. Comput., 74 (2018), pp. 826-850] for the relaxed PPA, and it was shown that the bound is optimal when the relaxation factor $\gamma$ lies in $[1,2)$. However, for other choices of $\gamma$, the bound obtained by Tao and Yuan is suboptimal. In this paper, we establish tight linear convergence bounds for any choice of $\gamma\in(0,2)$ and make the whole picture about optimal linear convergence bounds clear. These results sharpen our understandings to the asymptotic behavior of the relaxed PPA.

\bigskip

\end{abstract}

\section{Introduction}
\label{sc:intro}

The \emph{proximal point algorithm} (PPA) was pioneered by Moreau \cite{Mor65bsmf,Mor62} and  Martinet \cite{Mar70,Mart72}.
It was popularized in the optimization community by Rockafellar \cite{Roc76a,Roc76b}, mainly due to its global convergence under fairly mild conditions, fast asymptotic linear/superlinear convergence rate under certain regularity conditions, as well as its connections to a few classical optimization schemes. Ever since, PPA has been playing tremendously important roles in designing, analyzing and understanding of optimization algorithms. In this section,  we first briefly review PPA and then summarize the contributions and the organization of this paper.
Since the literature of  PPA has become so vast, a thorough overview is far beyond the focus of this paper.  Instead, we will keep our review short and succinct, mainly focusing on the closely related  theoretical works.

It was shown in \cite{Roc76b} that the landmark method of multipliers of Hestenes \cite{Hes69} and Powell \cite{Pow69} for constrained nonlinear optimization is a dual application of the PPA. The Douglas-Rachford operator splitting method \cite{DR56,LM79} is also an application of the PPA to a special splitting operator \cite{EB92}. Given the connection with the Douglas-Rachford splitting method revealed in \cite{Gabay83}, the influential alternating direction method of multipliers for linearly constrained separable convex optimization \cite{GM75,GM76} is also an application of the PPA.
In the general setting of maximally monotone inclusion problems, it was shown in \cite{Roc76a} that PPA, as well as some approximate variants of it, converge globally as long as the sequence of proximal parameters is bounded away from $0$ and the underlying problem admits a solution.
Recently, in the absence of regularity assumption, a nonasymptotic $O(1/N)$ convergence rate has been derived in \cite{HY15c}  for the Douglas-Rachford splitting method, a generalization of PPA to treating the sum of two operators, where $N$ denotes the iteration counter.
For strongly monotone operators, it was shown in \cite{Gabay83,EY14SIOPT} that PPA converges linearly when the sequence of proximal parameters keeps constant. In fact, asymptotically linear convergence rate is preserved by PPA as long as the inverse operator is Lipschitz continuous at the origin, a condition weaker than strong monotonicity, see \cite{Roc76a}. Recently, the linear convergence results in \cite{Roc76a} were generalized in \cite{TY18JSC} under the same regularity condition to a relaxed variant of PPA considered in \cite[Remark 2.3]{Gabay83}. Note that the relaxed PPA is also called generalized PPA in \cite{EB92}.
For minimizing a proper lower semicontinuous convex function and structured convex optimization, convergence analysis related to PPA is even more abundant. For example, a nonasymptotic $O(1/N)$ sublinear convergence rate measured by function value residual has been established in \cite{Gul91sicon}. Sublinear, as well as linear, convergence rates were derived in \cite{DavY17mor} under various regularity assumptions for several splitting methods that are closely related to PPA in the context of linearly constrained separable convex optimization. See also \cite{Gul92siopt,BT09} for some accelerated proximal-point-like methods designed for solving convex optimization problems by using Nesterov type acceleration technique \cite{Nest83}.

\subsection{Contributions}
In this paper, we further investigate the optimal linear convergence rate of relaxed PPA under the same regularity condition as in \cite{Roc76a,TY18JSC}, i.e., the inverse operator is Lipschitz continuous at the origin. When the relaxation parameter $\gamma$ lies in $[1,2)$,
the linear convergence factor obtained in \cite[Theorem 3.5]{TY18JSC} is tight. However, for other choices of $\gamma$, the results in \cite[Theorem 3.5]{TY18JSC} is suboptimal. The main contribution of this paper is to establish optimal linear convergence bounds for any choice of $\gamma\in(0,2)$, making the whole picture clear.

\subsection{Notation and organization}
In this paper, we reside ourselves in the $n$-dimensional Euclidean space $\R^n$, with inner product denoted by $\langle\cdot,\cdot\rangle$ and the induced norm $\|\cdot\| = \sqrt{\langle\cdot,\cdot\rangle}$, though all the analysis can be easily extended to any finite dimensional real Euclidean spaces.   Let $T: \R^n \rightrightarrows \R^n$ be a maximal monotone operator and $c>0$ be a scalar. The resolvent operator of $T$ is given by $J_{cT} := (I + cT)^{-1}$. The set of zeros of $T$ is denoted by $\text{zer}(T) := \{z^* \in \R^n \mid 0 \in T(z^*)\}$.

The rest of this paper is organized as follows. In Section \ref{sc:RPPA}, we specify the relaxed PPA and make our assumptions.
Linear convergence bounds are derived in Section \ref{sc:linear_factor}, followed by examples to show that the established bounds are nonimprovable in Section \ref{sc:bounds_optimal}. Finally, some concluding remarks are given in Section \ref{sc:conclusions}.

\section{Relaxed PPA}\label{sc:RPPA}
Let $T: \R^n \rightrightarrows \R^n$ be a set-valued maximal monotone operator. A problem of fundamental importance in convex analysis and convex optimization is to find a zero of $T$, i.e., find $z^*\in \R^n$ such that $0\in T(z^*)$.
Let $\gamma \in (0,2)$ and $\{c_k>0: k = 0, 1, 2, \ldots\}$ be a sequence of parameters.
Initialized at $z^0\in\R^n$, the relaxed (or generalized) PPA generates a unique sequence of points $\{z^k: k = 1, 2, 3, \ldots\}$ via
\begin{subequations}\label{PPA-scheme}
\begin{align}\label{PPA-1}
  \tilde{z}^k & := J_{c_kT}(z^k), \smallskip \\
\label{PPA-2}
   z^{k+1} & := (1-\gamma)z^k+\gamma \tilde{z}^k, \; \; k = 0, 1, 2, \ldots
\end{align}
\end{subequations}
See, e.g., \cite[Remark 2.3]{Gabay83} for an early reference about this relaxed PPA scheme.
It was shown in \cite{Minty62} that the resolvent operator of any maximal monotone operator is single-valued and everywhere defined.
Therefore, the scheme \eqref{PPA-scheme} is well defined.  Note that the original PPA corresponds to $\gamma\equiv 1$ in \eqref{PPA-scheme}.
Numerically, the introduction of the relaxation parameter $\gamma\in (0,2)$ can usually accelerate the original PPA, see, e.g.,
\cite[Section 2.3.1]{Bert82book} and \cite{FHLY15MPC} for numerical evidence, while theoretically, given the connections between PPA and the many classical algorithms reviewed in Section \ref{sc:intro}, introducing the relaxation parameter $\gamma$ usually inspires new algorithms. One particular example so inspired is the generalized alternating direction method of multipliers derived in \cite{EB92}.

Suppose that $\text{zer}(T)\neq \emptyset$ and $\{c_k>0: k=0,1,2,\ldots\}$ is bounded away from $0$.
It was first shown by Rockafellar \cite[Theorem 1]{Roc76a}  that the sequence $\{z^k: k=1,2,\ldots\}$ generated by the original PPA, or some well controlled approximate variants,  converges to some $z^*\in \text{zer}(T)$. If, in addition, $T^{-1}$ is Lipschitz continuous at $0$ with modulus $a\geq 0$, then the convergence rate is guaranteed to be linear eventually, where the linear convergence factor depends on the modulus $a$, among some others. Recently, these results were generalized in \cite{TY18JSC} to the relaxed PPA scheme \eqref{PPA-scheme} and its approximate variants.
In this paper, we carry out some further analysis following this line. For this purpose, we make the following assumptions.

\begin{assumption}\label{assumption-ck}
Assume that the sequence $\{c_k>0: k=0,1,2,\ldots\}$ is bounded away from $0$.
\end{assumption}
\begin{assumption}\label{assumption-Lip}
Assume that $T^{-1}$ is Lipschitz at $0$ with modulus $a\geq 0$, i.e., (i) $\text{zer}(T) = \{z^*\}$ is a singleton, where $z^*\in\R^n$, and (ii) for some $\tau >0$ it holds that $\|z - z^*\| \leq a \|w\|$ whenever $z \in T^{-1}(w)$ and $\|w\| \leq \tau$.
\end{assumption}
\begin{assumption}\label{assumption-Lip-k}
$k>0$ is an integer such that $\|(z^k-\tilde{z}^k)/c_k\| \leq \tau$ is satisfied, where $\tau >0$ is given in Assumption \ref{assumption-Lip}.
\end{assumption}

\begin{remark}
Note that $c_k$'s are algorithmic parameters and can be fully controlled. Assumption \ref{assumption-Lip} is a regularity assumption on $T$, which is weaker than strong monotonicity.
It follows from \cite[Theorem 3.2]{TY18JSC} that $\lim_{k\rightarrow\infty}\|z^k - \tilde{z}^k\| = 0$ provided that Assumption \ref{assumption-ck} is satisfied. As a result, Assumption \ref{assumption-Lip-k} can be fulfilled for sufficiently large $k$. In particular, if $c_k\equiv c>0$, then according to \cite[Theorem 3.1]{HY15c} it suffices to have Assumption \ref{assumption-Lip-k} satisfied whenever $k \geq {\hat k} := \frac{\|z^0-z^*\|^2}{\gamma(2-\gamma)\tau^2 c^2}$. The focus of this work is the optimal linear convergence factor at the $k$-th iteration, where $k$ is such that Assumption \ref{assumption-Lip-k} is fulfilled.
\end{remark}

Under Assumptions \ref{assumption-ck}-\ref{assumption-Lip-k}, it was shown in \cite[Theorem 3.5]{TY18JSC} that
\begin{equation}\label{TY-rate}
\|z^{k+1} - z^*\|^2 \leq \Bigl(1 - \min(\gamma,2\gamma-\gamma^2){c_k^2 \over a^2 + c_k^2}\Bigr)   \|z^{k+1} - z^*\|^2.
\end{equation}
In the case of $\gamma\in[1,2)$, it is also certificated in \cite{TY18JSC}  by an example that the bound \eqref{TY-rate} is tight.
We show in this paper that the bound \eqref{TY-rate} is suboptimal when $\gamma\in(0,1)$, and a tighter bound is then derived in this case, which is guaranteed to be optimal. Therefore, the whole picture about the optimal linear convergence bound of the relaxed PPA \eqref{PPA-scheme} is made clear.

\section{Linear convergence factor}\label{sc:linear_factor}
In this section, we establish linear convergence factor of the relaxed PPA \eqref{PPA-scheme} at the $k$-th iteration, where $k>0$ is sufficient large such that Assumption \ref{assumption-Lip-k} is satisfied.
It follows from \eqref{PPA-1} that  $\tilde{z}^k \in T^{-1}((z^k-\tilde{z}^k)/c_k)$. Therefore, Assumption \ref{assumption-Lip} implies that
\begin{equation}\label{Lip-invT-0}
\|\tilde{z}^k - z^*\| \leq a\|(z^k-\tilde{z}^k)/c_k\| = t_k  \|z^k-\tilde{z}^k\|,
\end{equation}
where $t_k := a/c_k$. Furthermore, it follows from  $    z^k-\tilde{z}^k \in c_k T\tilde{z}^k$, $0 \in c_k  Tz^*$ and the monotonicity of $T$ that
\begin{equation}\label{monotone-1}
   \langle z^k - \tilde{z}^k, \tilde{z}^k - z^* \rangle\geq 0.
\end{equation}
For simplicity, we denote
\begin{align}\label{def:minus-star}
  u^k := z^k - z^*, \text{~~}  u^{k+1} := z^{k+1} - z^* \text{~~and~~}   \tilde{u}^k := \tilde{z}^k - z^*.
\end{align}
Then,  \eqref{Lip-invT-0} and \eqref{monotone-1}  appear, respectively, as
\begin{equation}\label{mono-Lip-u}
 \|\tilde{u}^k\| \leq t_k  \|u^k-\tilde{u}^k\| \text{~~and~~}    \langle u^k - \tilde{u}^k, \tilde{u}^k \rangle\geq 0.
\end{equation}
Furthermore, \eqref{PPA-2} reads
\begin{equation}\label{relaxed-u}
u^{k+1}= (1-\gamma)u^k+\gamma \tilde{u}^k.
\end{equation}
For convenience, we define
\begin{eqnarray}\label{def:pars}
\left\{
\begin{array}{l}
  \mu_k := \frac{2\gamma(t_k^2+\gamma-1)}{t_k^2+1}, \;\;
  \nu_k := \frac{\gamma(2-\gamma)}{t_k^2+1}, \;\;
  \lambda_k    :=  \frac{\gamma(1-\gamma)}{t_k} + \frac{\gamma^2}{t_k+1}, \smallskip \\
  \varrho_{k,u} := 1-\frac{\gamma(2-\gamma)}{t_k^2+1}, \;\;
  \varrho_{k,l}  := \Bigl(1-\frac{\gamma}{t_k+1}\Bigr)^2  \text{~~and~~} \varrho_k := \max(\varrho_{k,u},\varrho_{k,l}).
\end{array}
\right.
\end{eqnarray}
It is easy to verify that
\begin{align*}
  \varrho_{k,u} - \varrho_{k,l}
 = \frac{2\gamma}{t_k+1} - \frac{\gamma^2}{(t_k+1)^2} - \frac{\gamma(2-\gamma)}{t_k^2+1} 
= \frac{2\gamma t_k (t_k^2 + \gamma - 1)}{(t_k+1)^2(t_k^2+1)},
\end{align*}
and thus
\begin{equation}\label{rhok}
\varrho_k = \max(\varrho_{k,u},\varrho_{k,l}) =
\left\{
  \begin{array}{ll}
    \varrho_{k,u}, & \hbox{if $t_k^2 + \gamma \geq 1$,} \smallskip \\
    \varrho_{k,l}, & \hbox{if otherwise.}
  \end{array}
\right.
\end{equation}
%
Following is our main convergence result on the linear convergence factor.
\begin{theorem}
  \label{thm:main}
Let $T: \R^n \rightrightarrows \R^n$ be a maximal monotone operator and $\gamma\in (0,2)$. Assume that $\{c_k: k=0,1,2,\ldots\}$ satisfies Assumption \ref{assumption-ck} and $T$ satisfies the regularity assumption \ref{assumption-Lip}.
Let $k>0$ be such that Assumption \ref{assumption-Lip-k} is satisfied.   Then, it holds that
\begin{equation}\label{thm-inequality}
\|u^{k+1}\|^2 \leq \varrho_k \|u^k\|^2.
\end{equation}
\end{theorem}
\begin{proof}
We separate the proof into two cases, (i) $t_k^2 + \gamma \geq 1$, and (ii) $t_k^2 + \gamma < 1$.

For case (i), it holds that $t_k^2 + \gamma \geq 1$. Then, it is elementary to verify that
\begin{align}\label{def:C}
\begin{pmatrix}
    C_{11}  \smallskip \\
    C_{12}  \smallskip \\
    C_{22}  \smallskip \\
\end{pmatrix}
:=
\begin{pmatrix}
    (1-\gamma)^2 + \nu_k t_k^2      \smallskip \\
    \gamma(1-\gamma) +  \mu_k/2 - t_k^2 \nu_k  \smallskip \\
    \gamma^2 - \mu_k + (t_k^2-1)\nu_k \smallskip \\
\end{pmatrix}
=
\begin{pmatrix}
    \varrho_{k,u} \smallskip \\
    0  \smallskip \\
    0  \smallskip  \\
\end{pmatrix}.
\end{align}
Since $t_k^2 + \gamma \geq 1$ and $\gamma \in (0,2)$, it is clear from \eqref{def:pars} that $\mu_k\geq 0$ and $\nu_k>0$. Furthermore,
$\nu_k \leq 1/(t_k^2+1) < 1$ and thus $0<\varrho_{k,u}<1$.
It thus follows from \eqref{mono-Lip-u},  \eqref{relaxed-u} and \eqref{def:C} that
\begin{align*} 
 \|u^{k+1}\|^2  \leq \; & \|u^{k+1}\|^2 + \mu_k \langle u^k - \tilde{u}^k, \tilde{u}^k \rangle   + \nu_k (t_k^2  \|u^k-\tilde{u}^k\|^2 - \|\tilde{u}^k\|^2) \smallskip \\
= \; & \|(1-\gamma)u^{k} + \gamma \tilde{u}^k\|^2 + \mu_k \langle u^k - \tilde{u}^k, \tilde{u}^k \rangle  + \nu_k (t_k^2  \|u^k-\tilde{u}^k\|^2 - \|\tilde{u}^k\|^2)  \smallskip \\
= \; & C_{11} \|u^k\|^2 + 2C_{12} \langle u^k, \tilde{u}^k\rangle + C_{22} \|\tilde{u}^k\|^2  \smallskip \\
= \; &  \varrho_{k,u}  \|u^k\|^2.
\end{align*}

For case (ii), it holds that $t_k^2 + \gamma < 1$. Then, it is elementary to verify that
\begin{align}\label{def:Cprime}
\begin{pmatrix}
    C_{11}'  \smallskip \\
    C_{12}'  \smallskip \\
    C_{22}'  \smallskip \\
\end{pmatrix}
:=
\begin{pmatrix}
    (1-\gamma)^2 +  \lambda_k  t_k^2 - \varrho_{k,l}   \smallskip \\
    \gamma(1-\gamma) - \lambda_k  t_k^2  \smallskip \\
    \gamma^2 + \lambda_k  (t_k^2-1)  \smallskip \\
\end{pmatrix}
=
\gamma(t_k^2 + \gamma - 1)
\begin{pmatrix}
    \frac{t_k}{t_k^2 + 1}  \smallskip \\
    {- 1 \over t_k + 1}  \smallskip \\
    {1 \over t_k}  \smallskip \\
\end{pmatrix}.
\end{align}
Since $t_k = a/c_k > 0$, $\gamma \in (0,2)$ and $t_k^2 + \gamma < 1$, it is easy to show from \eqref{def:pars}  that
$\lambda_k > 0$ and $0 < \varrho_{k,l} < 1$.
It thus follows from \eqref{mono-Lip-u}, \eqref{relaxed-u} and \eqref{def:Cprime} that
\begin{align*}
  \|u^{k+1}\|^2  \leq \;
& \|u^{k+1}\|^2   + \lambda_k  (t_k^2  \|u^k-\tilde{u}^k\|^2 - \|\tilde{u}^k\|^2)   \smallskip    \\
= \; & \|(1-\gamma)u^{k} + \gamma \tilde{u}^k\|^2  + \lambda_k  (t_k^2  \|u^k-\tilde{u}^k\|^2 - \|\tilde{u}^k\|^2)  \smallskip   \\
= \; & C_{11}' \|u^k\|^2 + 2C_{12}' \langle u^k, \tilde{u}^k\rangle + C_{22}' \|\tilde{u}^k\|^2 +   \varrho_{k,l}\|u^k\|^2  \smallskip  \\
= \; &  \gamma(t_k^2 + \gamma - 1) \Bigl( \frac{t_k}{t_k^2 + 1} \|u^k\|^2 - {2 \over t_k + 1} \langle u^k, \tilde{u}^k\rangle + {1 \over t_k} \|\tilde{u}^k\|^2  \Bigr) +   \varrho_{k,l}\|u^k\|^2  \smallskip  \\
\leq \; &  \gamma(t_k^2 + \gamma - 1) \Bigl( \frac{t_k}{(t_k + 1)^2} \|u^k\|^2 - {2 \over t_k + 1} \langle u^k, \tilde{u}^k\rangle + {1 \over t_k} \|\tilde{u}^k\|^2  \Bigr) +   \varrho_{k,l}\|u^k\|^2  \smallskip \\
= \; & \gamma(t_k^2 + \gamma - 1) \Bigl\|\frac{\sqrt{t_k}}{t_k + 1} u^k - {1 \over \sqrt{t_k}} \tilde{u}^k \Bigr\|^2 + \varrho_{k,l}\|u^k\|^2  \smallskip \\
\leq \; &  \varrho_{k,l}\|u^k\|^2,
\end{align*}
where the second and the third inequalities follow from $t_k>0$ and $t_k^2 + \gamma - 1 < 0$.

In summary, by considering the definition of $\varrho_k$ in \eqref{rhok}, we have proved \eqref{thm-inequality}.
\end{proof}

\section{The obtained bounds are optimal}\label{sc:bounds_optimal}
When $\gamma \in [1,2)$, the linear convergence bound given in \cite{TY18JSC} is tight. In this case, our bound is exactly the same as that given in  \cite{TY18JSC}. In this section, we first show that our bounds are better than those given in \cite{TY18JSC} when $\gamma \in (0,1)$ and then provide examples to show that the established bounds are always tight.

For convenience, we split the discussions into three cases, (i) $\gamma\in [1,2)$, (ii) $\gamma\in (0,1)$ and $t_k^2 + \gamma \geq 1$, and (iii) $\gamma\in (0,1)$ and $t_k^2 + \gamma < 1$. Furthermore, we denote the linear convergence bounds given in \cite{TY18JSC} by $\varrho^{\text{TY}}_k$, while the bounds established in this paper are denoted by $\varrho^{\text{GY}}_k$.
Recall that the bounds derived in  \cite{TY18JSC} are given in \eqref{TY-rate}, while ours are given in \eqref{thm-inequality}.

For case (i), the bounds are the same and are given by $\varrho^{\text{TY}}_k = \varrho^{\text{GY}}_k = 1 - \frac{\gamma(2-\gamma)}{t_k^2 + 1}$.

For case (ii), we have $\varrho^{\text{TY}}_k = 1 - \frac{\gamma}{t_k^2 + 1}$, $\varrho^{\text{GY}}_k =  1 - \frac{\gamma(2-\gamma)}{t_k^2 + 1}$ and thus
$\varrho^{\text{TY}}_k - \varrho^{\text{GY}}_k =  \frac{\gamma(1- \gamma)}{t_k^2 + 1} > 0$.

For case (iii), we have $\varrho^{\text{TY}}_k = 1 - \frac{\gamma}{t_k^2 + 1}$, $\varrho^{\text{GY}}_k  = \Bigl(1 - \frac{\gamma}{t_k + 1}\Bigr)^2$ and thus
\begin{align*}
\varrho^{\text{TY}}_k - \varrho^{\text{GY}}_k
&=  \frac{2\gamma}{t_k + 1} - \frac{\gamma^2}{(t_k + 1)^2} -\frac{\gamma}{t_k^2 + 1} \\
&=  \gamma \Bigl[ \frac{2t_k + 2 - \gamma}{(t_k + 1)^2} - \frac{1}{t_k^2 + 1} \Bigr] \\
&>  \gamma \Bigl[ 1 - \frac{1}{t_k^2 + 1} \Bigr] \\
& > 0,
\end{align*}
where the first inequality follows from $(t_k + 1)^2 = t_k^2 + 2t_k + 1 < 2t_k + 2 - \gamma$.
In summary, for cases (ii) and (iii) our bounds are always shaper than \eqref{TY-rate}, while for case (i) the bounds are identical and optimal.

Next, we provide examples to show that our bounds are always tight. In fact, Example \ref{ex1} given below for
the case $t_k^2 + \gamma \geq 1$ is the same as in \cite[Sec. 2.2]{TY18JSC}. We show that the worst-case bound $\varrho_{k,u}$ is attained by this example within the region $t_k^2 + \gamma \geq 1$, which is larger than $\gamma \in [1,2)$. Example \ref{ex2} given below for
the case $t_k^2 + \gamma < 1$  is one-dimensional. Both examples are linear.

\begin{example}\label{ex1}
Consider the case $t_k^2 + \gamma \geq 1$.
Let $a>0$ and define
\begin{equation}
T(z) = {1\over a} \begin{pmatrix}
                      0 & 1 \\
                      -1 & 0 \\
                  \end{pmatrix}
\begin{pmatrix}
    z_1 \\
    z_2 \\
\end{pmatrix}, \quad \forall \, z = \begin{pmatrix}
                       z_1 \\
                       z_2 \\
                   \end{pmatrix}\in\R^2.
\end{equation}
Clearly, $T$ is linear and single-valued everywhere.
Since the coefficient matrix defining $T$ is skew-symmetric, it holds that $\langle x - y, T(x) - T(y) \rangle = 0$ for any $x,y\in\R^2$.
In particular, $T$ is monotone. Since $T$ is also continuous, it follows that $T$ is maximally monotone.
Apparently, the unique solution of $0\in T(z)$ is $z^* := (0,0)^T$, i.e., the origin.
Furthermore, the inverse operator $T^{-1}$  of $T$ is given by
\begin{equation}
T^{-1}(z) = a \begin{pmatrix}
                      0 & -1 \\
                      1 & 0 \\
                  \end{pmatrix}
\begin{pmatrix}
    z_1 \\
    z_2 \\
\end{pmatrix}, \quad \forall \, z = \begin{pmatrix}
                       z_1 \\
                       z_2 \\
                   \end{pmatrix}\in\R^2,
\end{equation}
which is apparently Lipschitz continuous globally with modulus $a$. Let $c_k > 0$, $\tilde{z}^k = J_{c_kT}(z^k)$ and recall the notation defined in \eqref{def:minus-star}. Then, it follows from
$T \tilde{z}^k = (z^k - \tilde{z}^k)/c_k$ and $T(z^*) = 0$ that $\langle  u^k - \tilde{u}^k, \tilde{u}^k\rangle = 0$.
Furthermore, it is easy to verify that
\begin{equation}\label{ex1-ztilde}
\tilde{z}^k   =
{a\over a^2 + c_k^2}
\begin{pmatrix}
    a  & -c_k \\
    c_k & a \\
\end{pmatrix}
\begin{pmatrix}
    z^k_1 \\
    z^k_2 \\
\end{pmatrix}
\text{~~and~~}
\|\tilde{u}^k\|^2 = {a^2 \over a^2 + c_k^2} \|u^k\|^2.
\end{equation}
It follows from  \eqref{relaxed-u} that
\begin{align*}
  \|u^{k+1}\|^2
  & = \|(1-\gamma)u^k + \gamma \tilde{u}^k\|^2 \\
  & = (1-\gamma)^2 \|u^k\|^2 + 2\gamma(1-\gamma)\langle u^k, \tilde{u}^k\rangle + \gamma^2 \|\tilde{u}^k\|^2 \\
  & = (1-\gamma)^2 \|u^k\|^2 + 2\gamma(1-\gamma)\langle u^k-\tilde{u}^k, \tilde{u}^k\rangle + (2\gamma(1-\gamma)+\gamma^2) \|\tilde{u}^k\|^2 \\
  & = (1-\gamma)^2 \|u^k\|^2 + \gamma(2-\gamma)\|\tilde{u}^k\|^2 \\
  & = \Bigl((1-\gamma)^2  + \gamma(2-\gamma) \frac{a^2}{a^2 + c_k^2} \Bigr)  \|u^k\|^2 \\
  & = \varrho_{k,u}  \|u^k\|^2,
\end{align*}
where the fourth ``$=$" follows from $\langle  u^k - \tilde{u}^k, \tilde{u}^k\rangle = 0$, the fifth ``$=$" follows from \eqref{ex1-ztilde} and the last ``$=$" follows from $t_k = a/c_k$ and the definition of $\varrho_{k,u}$ in \eqref{def:pars}. As a result, the upper bound $\varrho_{k,u}$ is attained.
\end{example}

\begin{example}\label{ex2}
Consider the case $t_k^2 + \gamma < 1$.
Let $a>0$ and define $T(z) = z/a$ for $z\in \R$.
Apparently, $T$ is maximal monotone, and $T^{-1}(z) = a z$ for $z\in \R$.
Thus, $\text{zer}(T) = \{z^*\} = \{0\}$, and $T^{-1}$ is Lipschitz continuous globally with modulus $a$.
In fact, $T$ is strongly monotone with modulus $1/a$.
Let $c_k > 0$.
It is trivial to show that
\[\tilde{z}^k = J_{c_k T}(z^k) = (1 + c_k T)^{-1}(z^k) = {t_k z^k\over t_k + 1}, \quad \forall\, z^k\in\R,\]
where $t_k = a/c_k$. Therefore,  $\tilde{u}^k = {t_k u^k\over t_k + 1}$ and
\begin{align*}
\|u^{k+1}\|^2 = \|(1-\gamma)u^k + \gamma \tilde{u}^k\|^2 = \Bigl\| \Bigl(1-\gamma   +  {\gamma t_k \over t_k+1 }\Bigr) u^k \Bigr\|^2
= \varrho_{k,l} \|u^k\|^2.
\end{align*}
Here $\varrho_{k,l}$ is defined in \eqref{def:pars}. As a result, the upper bound $\varrho_{k,l}$ is attained.
\end{example}

The optimality of the established bounds are illustrated in Figure \ref{Fig:example}, in comparison with those given in \cite{TY18JSC}.
Recall that $t_k=a/c_k$.
\begin{figure}[h!]
\begin{center}
\begin{tikzpicture}[scale=1.6]
\path [fill=lightgray] (0, 1) -- (3, 1) -- (3, 2) -- (0, 2) -- (0, 1);
\draw[-latex] (0, 0) -- (0, 2.2) node [above]{$\gamma$};
\draw[-latex] (0, 0) -- (3, 0) node [right]{$({a}/{c_k})^2$};
\draw [-] (0, 1) -- (3, 1) node [right] {$\gamma=1$};
\draw [dashed] (0, 2) -- (3, 2) node [right] {$\gamma=2$};
\node [right] at (.2, 1.5) {$1-\frac{2\gamma -\gamma^2}{1+(a/c_k)^2}$, optimal};
\node [right] at (.2, 0.5) {$1-\frac{\gamma}{1+(a/c_k)^2}$, not optimal};
\node [right] at (1.2, -0.3) {  };
\end{tikzpicture}
\begin{tikzpicture}[scale=1.6]
\path [fill=lightgray] (0, 1) -- (1, 0) -- (3, 0) -- (3, 2) -- (0, 2) -- (0, 1);
\draw[-latex] (0, 0) -- (0, 2.2) node [above]{$\gamma$};
\draw[-latex] (0, 0) -- (3, 0) node [right]{$({a}/{c_k})^2$};
\draw [-] (-0.2, 1.2) -- (1.2, -0.2) node [right] {$\gamma+({a}/{c_k})^2=1$};
\draw [dashed] (0, 2) -- (3, 2) node [right] {$\gamma=2$};
\node [right] at (1, 1) {$\varrho_{k,u}$, optimal};
\node [right] at (0, 0.4) {$\varrho_{k,l}$,};
\node [right] at (0, 0.1) {optimal};
\end{tikzpicture}
\end{center}
\caption{Illustration of the bounds. Left: results in \cite{TY18JSC}. Right: results of this work.}\label{Fig:example}
\end{figure}
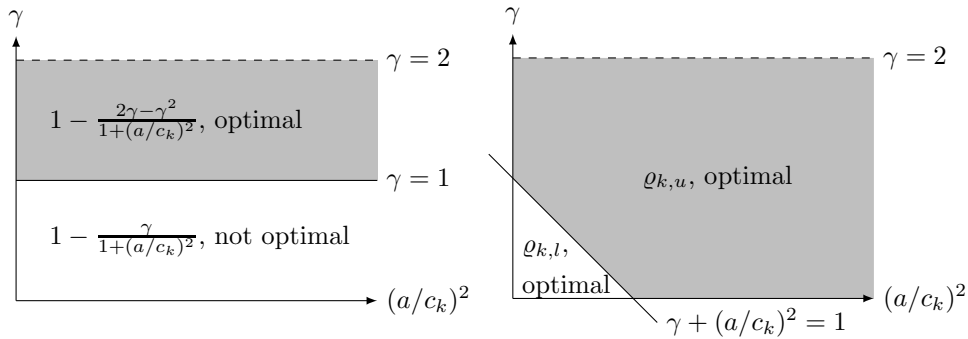

\section{Concluding remarks}\label{sc:conclusions}
In this paper, we have investigated linear convergence rate of the relaxed PPA under the regularity condition that $T^{-1}$ is Lipschitz continuous at the origin. We have established tight linear convergence bounds for all choices of $\gamma\in(0,2)$. In comparison,
the bounds given in \cite{TY18JSC} are optimal only for $\gamma\in[1,2)$.
Our proof of Theorem \ref{thm:main} is constructive, and the discovery of the separating line $t_k^2 + \gamma = 1$ makes the whole picture about optimal linear convergence bounds clear.
The monotonicity of $\{c_k: k=0,1,2,\ldots\}$, as required in \cite{Roc76a}, is in fact irrelevant here since our result on the linear convergence factor is a one-step analysis. In the case of $c_k\equiv c>0$, the bound $\varrho_k$ in \eqref{rhok} becomes a constant $\varrho\in (0,1)$. As a result, \eqref{thm-inequality} holds with $\varrho_k\equiv \varrho$ for all $k$ greater than some ${\hat k}>0$, and, furthermore, it is tight.
Finally, we point out that the extension of our proof to inexact variants of the relaxed PPA \eqref{PPA-scheme} with \eqref{PPA-1} being replaced by approximate evaluation of $J_{c_kT}$ as in \cite{Roc76a} seems within reach.


\def\cprime{$'$}

\end{document}